\RequirePackage{fix-cm}
\documentclass[smallextended]{svjour3}       
\smartqed  
\usepackage{graphicx}
\usepackage{amsmath}
\usepackage{amssymb}
\usepackage{ntheorem}
\usepackage[misc]{ifsym}

\newtheorem*{theorem*}{Theorem}

\begin{document}

\title{Recombination processes and non-linear Markov chains
}


\author{S.A.~Pirogov \and A.N.~Rybko \and A.S.~Kalinina \and M.S.~Gelfand
}


\institute{S.A.~Pirogov$^1$\and A.N.~Rybko$^1$ \and A.S.~Kalinina$^1$(\Letter) \and M.S.~Gelfand$^{1,2}$\at
              $^1$A.A. Kharkevich Institute for Information Transmission Problems, RAS, Moscow, Russia\\
              $^2$M.V. Lomonosov Moscow State University, Department of Bioengineering and Bioinformatics, Moscow, Russia\\
              \email{as.kalinina@gmail.com}           
}

\date{Received: date / Accepted: date}

\maketitle

\begin{abstract}
Bacteria are known to exchange genetic information by horizontal gene transfer. Since the frequency of homologous recombination  depends on the similarity between the recombining segments, several studies examined whether this could lead to the emergence of subspecies. Most of them simulated fixed-size Wright--Fisher populations, in which the genetic drift should be taken into account.
Here, we use non-linear Markov processes to describe a bacterial population evolving under mutation and recombination. We consider a population structure as a probability measure on the space of genomes. This approach implies the infinite population size limit, thus the genetic drift is not assumed. We prove that under these conditions the emergence of subspecies is impossible.
\end{abstract}

\section{Introduction}
\label{intro}
Bacterial speciation differs from that in animals or plants, where the natural limitations on breeding exist, due to the lack of sexual reproduction in  prokaryotes. Nonetheless, bacteria are capable of obtaining genetic information from sources other then their maternal cells. 

Several species can acquire DNA directly from the environment. 
This process is called natural transformation. Approximately 1\% of bacterial species are known to have this ability, i.e. are competent (Jonas et al. 2001; Thomas and Nielsen 2005). Many of these species are not permanently competent, their ability to uptake DNA being induced by many factors such as stress and starvation.

Other mechanisms for horizontal DNA transfer are conjugation and transduction. Non-competent species, such as \emph{Escherichia coli}, acquire DNA from other bacteria via conjugative plasmids (conjugation) or phages (transduction) (Arutyunov and Frost 2013; Weinbauer and Rassoulzadegan 2004; Dixit et al. 2014).

Following uptake, DNA can be used by a cell as food or integrated in the genome by homologous recombination. As it has been demonstrated \emph{in vitro}, the probability of successful homologous recombination depends, firstly, on the similarity of the recombining segments, and, secondly, on their length (Shen and Huang 1986; Majewski and Cohan 1999; Vuli\'c et al. 1997).

Homologous recombination plays a major role in shaping bacterial species (Chan et al. 2011; Yahara et al. 2012). The process of homologous recombination is believed to be more intensive within bacterial species than between them due to higher similarity of genomes and common environment (Skippington and Ragan 2012). Thus, bacterial species should be homogeneous, but, in fact, they often form stable subspecies or phylogenetic groups (Guttman and Dykhuizen 1994; Chaudhuri and Henderson 2012), which may be considered as the early stage of the bacterial speciation.

The emergence of clusters of genomes as a result of  niche specialization, geographical isolation or selective pressure is possible (Polz et al. 2013; Koeppel et al. 2013; Cheng et al. 2015), but it is not obvious whether clusters may emerge in neutral models with only the mutation and homologous recombination processes.

Previous studies generated no consensus on the emergence of stable clusters of genomes in neutral models. Falush et al. (2006) have shown that stable isolated clusters emerge in the neutral model with appropriate values of the mutation rate to the recombination rate ratio and other parameters of simulation . More general simulations showed that the  emergence of clusters is likely in the absence or low rate of homologous recombination, where the clonal populations form clusters, whereas the high rate of homologous recombination acts like a cohesive force (Fraser et al. 2007). 

Furthermore, it has been analytically shown that distinct populations may be maintained by the mutation and homologous recombination processes without other factors (Doroghazi and Buckley 2011). However, in this study the distance between two populations was defined as the mean distance between all pairs of genomes, so if two similar populations with high variance formed one cluster, they still had non-zero distance between them.

An experimental study on dependence of recombination rate on sequence similarity \emph{in vivo} (Bao et al. 2014) demonstrated that if the recombination rate fell as sequence divergence increased, no clear-cut genomic boundaries between species emerged. On the other hand, such boundaries are observed (Tang et al. 2013), and the process of uptake exogenous DNA \emph{in vivo} differs significantly from that \emph{in vitro}. 
Understanding of the bacterial population behaviour in the neutral model entails understanding of bacterial subspecies isolation and reduction of homologous recombination between them (Ellegaard et al. 2013).

Here we consider the possibility of phylogroup emergence in the neutral model due to solely mutations and recombination. In (Lyubich 1971) this situation was considered for a diploid population, and for this model convergence to equilibrium was proved, but dependence of recombination rate of sequence similarity was not considered. This property of homologous recombination is essential in all studies on bacterial speciation in the neutral model.
Finally, properties of some models of recombination process were studied in (Baake 2011a, Baake 2011b).

Here we define a bacterial population as a set of genomes that continuously exchange genetic information via homologous recombination. For simplicity, we assume that the genomes can be aligned throughout their entire length, so that coordinates in a genome completely define the homologous region in another genome. Below, after giving formal definitions, we write a differential equation that describes a population under mutation and recombination processes in terms of probability measures on the space of genomes, and examine its fixed points. The equation describes the behavior of a population in the infinite size limit. For the finite size there is no closed system of equations for the average fractions of different genomes in the population. Our main tool, the monotonicity of the entropy, was used in other situations in (Kun and Lyubich 1980) and by L.Boltzmann in statistical physics.
Non-trivial behavior of the equation solutions would correspond to a complex population structure that hypothetically could emerge in this model.
 
\section{Results}
\label{sec:1}
Let $K$ be a finite alphabet (a set of nucleotides) and let a genome $x$ be a word of length $n$  over it. We consider two transformations of a genome:

1) mutations, when one letter changes to another $x_i\rightarrow y_i$, $i \in \Lambda=\{1,\ldots,n\}$, the mutation matrix is supposed to be irreducible, i.e. it is possible to get any letter from any other by several mutations;

2) homologous recombination, when a substring $x_I$ changes with a certain probability to substring $y_I$ with the same coordinates from another genome $y$. Here $I$ is any subset of $\Lambda$, $I \subset  \Lambda$ (hence, this definition is more general than in biology, where $I$ should be an interval in $\Lambda$).

The fundamental difference between these two transformations is that mutations occur in a genome independently of other genomes.
Formally, for any position $i$ in a genome, for any nucleotides $a,b\in K, a\neq b$, there exists a probability of transition  $a\rightarrow b$, denoted by $\alpha_i(a,b)$.
This means that for a small period of time $dt$ the probability of mutation of nucleotide $a$ to nucleotide $b$ approximately equals $\alpha_i(a,b)dt$.

Homologous recombination results from interaction of genomes in the space of genomes $X$. 
The recombination probability depends on the distribution of genomes in the space $X$ and a function $\varphi(x_I,y_I)$ which defines similarity between genomes $x$ and $y$ on substring $I$. 
This function is symmetric and non-negative. The distribution of genomes in $X$ is characterized by the probability distribution $\mu_\Lambda(x)$. 
Thus, the probability $P_\mu^{(I)}(x\rightarrow y)$ of substitution of a substring $x_I$ in genome $x$ to substring $y_I$ from genome $y$ equals $\varkappa\varphi(x_I,y_I)\mu_I(y_I)dt$ up to terms of order $(dt)^2$, $\varkappa$ is a constant and $\mu_I(y_I)$ is the marginal distribution, i.e. the probability distribution of substring  $y_I$.

Importantly, the probability of recombination on a substring $I$ in a genome depends on the probability distribution of all genomes in $X$. Such processes are called continuous-time non-linear Markov processes in the sense of McKean (1996) (i.e. Markov processes whose generator depends on a measure). The dependence of the probability distribution $\mu(x)$ on time is described by a non-linear differential equation
\begin{equation}\label{equ}
\begin{split}
\frac{d\mu_\Lambda(x_\Lambda)}{dt} = \sum_i\sum_{y_i}\left(
\alpha_i(y_i,x_i)\mu_\Lambda(x_{\Lambda\setminus i},y_i)-
\alpha_i(x_i,y_i)\mu_\Lambda(x_\Lambda)\right)+\\
\varkappa\sum_I\sum_{y_I}\left(
\varphi(y_I,x_I)\mu_I(x_I)\mu_\Lambda(x_{\Lambda\setminus I},y_I)-
\varphi(x_I,y_I)\mu_I(y_I)\mu_\Lambda(x_\Lambda)\right),
\end{split}
\end{equation}
(unlike the linear Kolmogorov forward equation for usual Markov processes).
The right-hand side of this equation is the sum of the following terms:

1) linear terms for mutations;

2) non-linear terms for substrings $I\subset\Lambda$, where recombination is possible.

As proved by Kurtz (Kurtz 1970; Ethier and Kurtz 1986), this equation is exact in the infinite size population limit.

Here we prove that if only mutation and recombination processes are considered, and the similarity function $\varphi(x_I,y_I)$ is symmetric, then for all values of other parameters, such as the ratio of the intensity of mutations and recombinations, or an initial distribution of genomes, there is a unique fixed point. This fixed point, as we show below, is the stationary distribution $q_\Lambda$ for the pure mutation process (the process without recombinations).

\begin{theorem*}
Equation (\ref{equ}) has a unique fixed point $q_\Lambda$ and all trajectories of (\ref{equ}) $\mu_\Lambda(t)\rightarrow q_\Lambda$ as $t\rightarrow \infty$.
\end{theorem*}

\noindent\textbf{Note} From the convergence of trajectories it follows that for a population consisting of $N$ individual bacteria (in the stationary state), the fraction $f_N(x)$ of bacteria having   genome $x$ converges in probability to $q_\Lambda(x)$ when $N$ tends to infinity, see (Liggett 2005, Chapter 1). 
It follows also that $Ef_N(x)\rightarrow q_\Lambda(x)$ when 
$N \rightarrow \infty$.

We have no detailed information about the dependencies between genomes for finite $N$. However, in the limit of infinite population size, genomes sampled from the population are independent. The asymptotic independence also follows from the convergence of fractions $f_N(x)$ (Pirogov and Petrova 2014).\\

To prove the Theorem, we use the Lyapunov method. The Lyapunov function is the Kullback-Leibler divergence (relative entropy) of $\mu_\Lambda$ with respect to $q_\Lambda$. 

Consider the mutation and recombination processes separately. As mentioned above, if $dt$ is small, the recombination process on the substring $I$ can be described as a non-linear discrete time Markov chain on the space $X$ with transition probabilities 
\begin{equation}\label{P}
P^{(I)}_\mu(x\rightarrow y)=\varkappa\delta(x_{\Lambda\setminus I},y_{\Lambda\setminus I})\varphi(x_I,y_I)\mu_I(y_I)dt
\end{equation}
for $y\neq x$, and $P^{(I)}_\mu(x\rightarrow x) = 1 - \sum_{y\neq x}P^{(I)}_\mu(x\rightarrow y)$. Here $\delta$ is the Kronecker delta.

Obviously, for this Markov chain the probability distribution
$$
\hat\mu_\Lambda(x_\Lambda)=\mu_{\Lambda\setminus I}(x_{\Lambda\setminus I})\mu_I(x_I)
$$
 is an invariant measure 
(here it is important, that the similarity function $\varphi(x_I,y_I) = \varphi(y_I,x_I)$ is symmetric). 
Moreover, any measure $\nu_\Lambda(x)$ on the space $X$ with marginal distributions $\mu_{\Lambda\setminus I}(x_{\Lambda\setminus I})$ and $\mu_I(x_I)$ turns to a measure $\nu_\Lambda P_\mu^{(I)}(x)=\sum_{y\in X}\nu_\Lambda(y)P^{(I)}_\mu(y\rightarrow x)$ having the same marginal distributions. 
So for the given measure $\mu_\Lambda$, the matrix $P_\mu^{(I)}$ is the transition matrix of the usual (linear) Markov chain with the invariant measure $\hat{\mu}_\Lambda$.

We now use an inequality for finite Markov chains, although it is more general in (Yosida 1940; Yosida 1965).

\begin{lemma}
Let $P$ be a stochastic matrix, i.e. matrix $P_{xy}$ such that
$P_{xy}\geqslant 0$ and $\sum_y P_{xy} = 1$, and let $\hat{\mu}$
be an invariant probability measure, $\hat\mu = \hat\mu P$. Suppose $\hat\mu(x)>0$ for any $x$. Then, for any probability measure $\mu$, 
\begin{equation}\label{I0}
\sum_x\left(\ln \frac{\left(\mu P\right)(x)}{\hat\mu (x)}\right)(\mu P)(x)\leqslant\sum_x\left(\ln\frac{\mu(x)}{\hat\mu(x)}{}\right)\mu(x)
\end{equation}
\end{lemma}
(Here as always $0 \ln 0 = 0$).

In our case, $\hat\mu_\Lambda(x_\Lambda) = \mu_{\Lambda\setminus I}(x_{\Lambda\setminus I})\mu_I(x_I)$, so $\ln \hat\mu=\ln \mu_I(x_I)+\ln \mu_{\Lambda\setminus I}(x_{\Lambda\setminus I})$  is a sum of functions depending only on $x_I$ and $x_{\Lambda\setminus I}$, respectively. Since $P^{(I)}_\mu$, acting on the measure $\mu_\Lambda$, retains marginal distributions of $x_I$ and $x_{\Lambda\setminus I}$, it follows that
\begin{equation}\label{E0}
\sum_x\left(\ln \hat\mu_\Lambda(x)\right)(\mu_\Lambda P_\mu^{(I)})(x)=\sum_x\left(\ln \hat\mu_\Lambda(x)\right)\mu_\Lambda(x)
\end{equation}

Finally, the Lemma yields the entropic inequality
\begin{equation}\label{E}
\sum_x\left(\ln\left(\mu_\Lambda P_\mu^{(I)}\right)(x)\right)\left(\mu_\Lambda P_\mu^{(I)}\right)(x) \leqslant
\sum_x\left(\ln \mu_\Lambda(x)\right)\mu_\Lambda(x)
\end{equation}

Now consider mutations. It is supposed that transition intensities 
$\alpha_i(a,b)$ define a connected continuous-time Markov chain
on alphabet $K$, so it is possible to pass from any $a$ to any $b$ in several steps.
By definition, $\alpha_i(a,a)=-\sum_{b\neq a}\alpha_i(a,b)$. 
Matrix $A_i=\left(\alpha_i(a,b),a,b\in K\right)$ is called the infinitesimal matrix of a time-continuous Markov chain. 
It is well known that for such chain there exists a unique invariant distribution $q_i(a), a \in K$ and $q_i(a)>0$. 
In terms of matrix $A_i$ this means that 
$q_iA_i=0$ (by definition $\left(q_iA_i\right)(x)=\sum_y q_i(y)\alpha_i(y,x)$). 

To describe mutations in any arbitrary position in the genome consider the following continuous-time Markov chain. Let $A_\Lambda = (a_\Lambda(x,y),x,y\in X)$ be the infinitesimal matrix, $a_\Lambda(x,y) = \sum_i\delta\left(x_{\Lambda\setminus i},y_{\Lambda\setminus i}\right)\alpha_i(x_i,y_i)$. The invariant distribution of the chain, defined by matrix $A_\Lambda$, is 

$$
q_\Lambda(x_\Lambda)=\prod_i q_i(x_i)
$$

Obviously, this chain is connected on the space $X$.

Finally, we use a general statement about the entropy monotonicity that is well known from the folklore and from results of (Batishcheva and Vedenyapin 2005) as a special case.

\begin{lemma}
Let $\alpha_{xy}$ be the transition intensities of a connected finite continuous time Markov chain and let $q_x$ be its stationary distribution. Then the relative entropy $D(p|q) = \sum_x p(x)\ln \frac{p(x)}{q(x)}$ is strictly decreasing and, furthermore, its derivative is strictly negative along the trajectory of the Kolmogorov forward equation $\dot p = pA$, where $A$ is the  infinitesimal matrix of the considered Markov chain. 
\end{lemma}

\begin{proof}
(for the reader's convenience).

Let $p(t)$ be the solution of the Kolmogorov forward equation and denote  $\frac{p_x}{q_x}$ by $f_x$, then the derivative $\frac{d}{dt}D(p(t)|q)$ can be written as 
$$
\frac{dD}{dt}=-\sum_{x,y}\left(\frac{f_x}{f_y}\ln \frac{f_x}{f_y}-\frac{f_x}{f_y}+1\right)q_x\alpha_{xy}
f_y
$$
Obviously, after removing parentheses, the two last terms in this formula cancel out, but they are needed to prove monotonicity. The expressions in parentheses are non-negative and, as the Markov chain is connected, they can be simultaneously
equal to $0$ only if $f_x=f_y$ for all $x,y$, i.e. if the distributions $p$ and $q$ are the same. \qed

\end{proof}

We now collect the properties of the mutation and homologous-recombination processes described above.

1) For the recombination process on substring $I$
$$
H(\mu_\Lambda) = \sum_x \left(\ln \mu_\Lambda (x)\right) \mu_\Lambda (x)
$$
monotonically (maybe, non-strictly) decreases, so its time derivative is non-positive.

2) For the same process, the value $\sum_x (\ln q_\Lambda)\mu_\Lambda(x)$ does not change, because this logarithm is the sum of functions of $x_I$ and 
$x_{\Lambda\setminus I}$, and as shown above, the means of such functions remain constant.

3) Hence, the relative entropy
$$
D(\mu_\Lambda|q_\Lambda)=\sum_x \left(\ln \frac{\mu_\Lambda(x)}{q_\Lambda(x)}\right)\mu_\Lambda(x)
$$
also has a non-positive derivative.

4) For the mutation process, the relative entropy $D(\mu_\Lambda|q_\Lambda)$ has a strictly negative derivative.\\

The right-hand side of equation (\ref{equ}) consists of the terms for the recombination process on all substrings $I$, and for the
mutation process. Since the relative entropy has a non-positive derivative by equations for the recombination process and a strictly negative derivative for the mutation process, its derivative by equation (\ref{equ}) is strictly negative, if $\mu_\Lambda\neq q_\Lambda$. This means, that the relative entropy strictly decreases along the trajectory of equation (\ref{equ}) and this equation has a unique fixed point $q_\Lambda$. As noted above, fixed points of equation (\ref{equ}) correspond to different population structures. A unique fixed point $q_\Lambda$ depends only on the infinitesimal matrix for the mutation process, so it gives us a population without non-trivial structure; if $q_i$ do not depend on $i$, then the probability of a genome depends only on its nucleotide composition. Note that if the similarity function $\varphi(x_I,y_I)$ and  the constant $\varkappa$ depend on time, it would not affect the calculations above.

\section{Discussion}
\label{sec:2}
Our results are consistent with simulations in (Fraser et al. 2007) with one difference. When the recombination rate is low, mutations lead to an increase of variance in a mostly clonal population, otherwise clusters are mixed by recombination. However, in that setting the genetic drift may cause speciation by chance as in (Falush et al. 2006), if the recombination rates vary appreciably between members of the population. 

Here we do not examine the behavior of a system in time, so we cannot claim that clusters may not emerge transiently, but we demonstrate there is no force that could maintain them. The model is general, as it accommodates various types of dependence of the recombination rate on sequence similarity, e.g. log-linear (Vuli\'c et al. 1997). However, the symmetry of the function $\varphi$ is a strong restriction and it seems to be weakly applicable to natural populations. For example, in the case of conjugative plasmids, the probabilities of DNA transfer between $F^+$ and $F^-$ cells in different directions are not equal (Arutyunov and Frost 2013), and hence $F^+$ genomes may form clusters.

We have examined the possibility that homologous recombination could drive bacterial speciation and have demonstrated that it could not. The possibility that stable clusters could emerge only due to the recombination rates, dependence on sequences similarity, is directly rejected. The homologous recombination may affect the rate  of spectiation, but it could not be the reason of it by itself. 
Mechanisms such as environmental isolation or fitness landscape are probable causes of bacterial speciation. A significant role may be played by phages. For example, in \emph{E. coli}, transduction seems to be responsible for most of the recombination events, because in this species conjugation and transformation processes are ineffective (Dixit et al. 2014).

\section{Appendix}
\label{sec:3}
\textbf{Definition}
Define $I/I$-recombination as the transform of pair of genomes   $x= ( x_I, x_{\Lambda\setminus I} )$  and $y=(y_I, y_{\Lambda\setminus I})$   to the pair $\tilde{x}=(y_I, x_{\Lambda\setminus I})$  and $\tilde{y}= (x_I, y_{\Lambda\setminus I})$.\\

Consider a finite population of genomes with mutations defined as above and $I/I$-recombinations  of pairs of genomes. The $I/I$-recombination rate of the pair $x,y$  equals  to $\varkappa \varphi(x_I, y_I)$.  As before, we assume  the function $\varphi$ to be symmetric. Then, in the Kurtz limit (infinite size limit), the dynamics of this population is  governed by equation (\ref{equ}). 
Note that $I/I$-recombination differs from $I$-recombination considered above. $I$-recombination can be defined as a transform of a genome $x=(x_I, x_{\Lambda\setminus I})$  to the genome  $\tilde{x}=(y_I, x_{\Lambda\setminus I})$  without any change in the genome $y$. However, kinetic equation (\ref{equ}) is the same for both cases, but now we may consider equation (\ref{equ}) as the Boltzmann equation for ``molecules" which exchange ``the parts of their velocities (genomes)" due to random collisions ($I/I$-recombinations)  and have some random ``free motion" (mutations).
 It is known that for the Boltzmann  equation of this type, the Kullback-Leibler entropy (in fact, negative entropy) has a strictly negative derivative in time (Batishcheva and Vedenyapin 2005; Pitaevskii and Lifshic 1981). The derivative is zero only in the fixed point. This fixed point is the invariant distribution $q_\Lambda$  for the ``free motion", i.e. for the pure mutation process. The detailed balance condition for collisions (Malyshev et al. 2004; Malyshev and Pirogov 2008) is satisfied due to the symmetry of $\varphi$. So we again see that the Kullback-Leibler relative entropy is the Lyapunov function for system (\ref{equ}). And  so again any trajectory of (\ref{equ}) converges to the fixed point $q_\Lambda$.

\begin{acknowledgements}
This study was supported by a grant from the Russian Science
Foundation (14-24-00155).
\end{acknowledgements}


\end{document}